\pdfoutput=1 % FOR ARXIV
\documentclass[a4paper,egregdoesnotlikesansseriftitles,final]{scrartcl}
\usepackage[utf8]{inputenc}
\usepackage[british]{babel}
\usepackage[T1]{fontenc}
\usepackage{lmodern}
\usepackage{microtype}
\usepackage{amsfonts,amsthm,amssymb,amsmath}
\usepackage{thm-restate}
\usepackage{mathtools}
\newcommand{\lowfwd}[3]{{\mathop{\kern0pt #1}\limits^{\kern#2pt\raise.#3ex
			\vbox to 0pt{\hbox{$\scriptscriptstyle\rightarrow$}\vss}}}}
\newcommand{\lowbkwd}[3]{{\mathop{\kern0pt #1}\limits^{\kern#2pt\raise.#3ex
			\vbox to 0pt{\hbox{$\scriptscriptstyle\leftarrow$}\vss}}}}
\newcommand{\fwd}[2]{{\lowfwd{#1}{#2}{15}}}

\newcommand{\vS}{{\hskip-1pt{\fwd S3}\hskip-1pt}}
\newcommand{\vSstar}{{\mathop{\kern0pt S\lower-1pt\hbox{$^*$}}\limits^{\kern2pt
			\vbox to 0pt{\hbox{$\scriptscriptstyle\rightarrow$}\vss}}}}
\newcommand{\vSdash}{{\mathop{\kern0pt S\lower-1pt\hbox{${}% logically \v(e')
				\scriptstyle'$}}\limits^{\kern2pt\raise.1ex
			\vbox to 0pt{\hbox{$\scriptscriptstyle\rightarrow$}\vss}}}}
\usepackage{xcolor}
\definecolor{skyblue}{HTML}{3465a4}
\usepackage{graphicx}
\usepackage{tikz}
\usetikzlibrary{arrows}
\usepackage[colorlinks,unicode,bookmarks,pdfusetitle,linkcolor={red!60!black},citecolor={green!60!black},urlcolor={blue!60!black}]{hyperref}
\usepackage[capitalize,nameinlink]{cleveref}
\crefname{enumi}{}{}
\crefname{property}{property}{properties}
\crefname{LEM}{Lemma}{the Lemmas}
\usepackage[shortlabels]{enumitem}
\usepackage{thmtools, thm-restate} %For duplicating theorem numbers.
\usepackage{cite}
\usepackage[abbrev]{amsrefs}
\usepackage{doi}
\usepackage{appendix}
\renewcommand{\PrintDOI}[1]{\doi{#1}}

\newtheorem{THM}{Theorem}%[section]
\newtheorem{LEM}{Lemma}[section]
\newtheorem{COR}[LEM]{Corollary}
\newtheorem{PROP}[LEM]{Proposition}

\theoremstyle{definition}
\newtheorem{EX}[LEM]{Example}

\newcommand{\abs}[1]{\lvert#1\rvert}

\newcommand{\menge}[1]{\left\{#1\right\}}

\newcommand{\tn}[1]{\textnormal{#1}}

\renewcommand{\phi}{\varphi}

\newcommand{\N}{\mathbb{N}}

\newcommand{\sub}{\subseteq}
\newcommand{\sm}{\smallsetminus}
\newcommand{\es}{\emptyset}
\renewcommand{\le}{\leqslant}
\renewcommand{\ge}{\geqslant}

\newcommand{\cP}{\mathcal{P}}

\newcommand{\cR}{\mathcal{R}}

\def\restricts{\!\restriction\!}
\def\proj{\restricts}

\DeclareMathOperator{\dom}{dom}
\DeclareMathOperator{\coh}{coh}

\colorlet{colorA}{orange!70!white}
\colorlet{colorB}{cyan!60!black}

\title{Ends as tangles}
\author{Jakob Kneip} %

\begin{document}
	
\maketitle

\begin{abstract}
	Every end of an infinite graph $ G $ defines a tangle of infinite order in~$ G $. These tangles indicate a highly cohesive substructure in the graph if and only if they are closed in some natural topology.
	
	We characterize, for every finite $ k $, the ends $ \omega $ whose induced tangles of order $ k $ are closed. They are precisely the tangles $ \tau $ for which there is a set of $ k $ vertices that decides $ \tau $ by majority vote. Such a set exists if and only if the vertex degree plus the number of dominating vertices of $ \omega $ is at least $ k $.
\end{abstract}

\section{Introduction}\label{sec:infinite_ends_intro}

Our first object of study in infinite tangle theory are tangles in infinite graphs. %Tangles were introduced by Robertson and Seymour in~\cite{GMX} and have since become one of the central objects of study in the field of graph minor theory. So far most of the theory of tangles deals with tangles in finite graphs.
In~\cite{EndsAndTangles}, it was shown how the set~$ \Theta $ of tangles of infinite order of an arbitrary infinite graph can be used to compactify that graph, much in the same way as the set~$ \Omega $ of ends of a connected locally finite graph can be used to compactify it. Indeed, if a graph~$ G $ is connected and locally finite, these compactifications~$ \abs{G}_\Theta $ and~$ \abs{G}_\Omega $ of~$ G $ coincide. This is because every end~$ \omega $ of an infinite graph~$ G $ induces a tangle~$ \tau=\tau_\omega $ of order~$ \aleph_0 $ in~$ G $, and for locally finite connected~$ G $ the map~$ \omega\mapsto\tau_\omega $ is a bijection between the set~$ \Omega $ of ends of~$ G $ and the set~$ \Theta $ of its~$ \aleph_0 $-tangles. (Graphs that are not locally finite have~$ \aleph_0 $-tangles that are not induced by an end.)

In~\cite{EndsAndTangles}, a natural topology on the set $ {\vS=\vS_{\aleph_0}(G)} $ of separations of finite order of~$ G $ was defined. A tangle~$ \tau $ induced by an end of~$ G $ is a closed set in this topology if and only if~$ \tau $ is defined by an~$ \aleph_0 $-block in~$ G $, that is, if there is an~$ \aleph_0 $-block~$ K $ in~$ G $ with~$ K\sub B $ for all separations~$ (A,B) $ in~$ \tau $.

Our research expands on this latter result. Every end~$ \omega $ of a graph induces not only a tangle of infinite order in~$ G $, but for each~$ k\in\N $ the end~$ \omega $ induces a~$ k $-tangle in~$ G $. The set~$ \vS_k $ of all separations~$ (A,B) $ of~$ G $ with~$ \abs{A\cap B}<k $ is a closed set in~$ \vS $, and thus if the tangle~$ \tau $ induced by~$ \omega $ in~$ G $ is a closed set in~$ \vS $, the~$ k $-tangle~$ \tau\cap\vS_k $ induced by~$ \omega $ will be closed as well. However, it is possible that a tangle~$ \tau $ in~$ G $ of infinite order fails to be closed in~$ \vS $, while its restrictions~$ \tau\cap\vS_k $ to~$ \vS_k $ are closed for some, or even all,~$ k\in\N $. In this paper we characterize the ends of~$ G $ by the behaviour of their tangles, as follows: We show that, for an end~$ \omega $ and its induced tangle~$ \tau $, the restriction~$ \tau\cap\vS_k $ to~$ \vS_k $ is a closed set in~$ \vS $ if and only if
\[ {\deg(\omega)+\dom(\omega)\ge k}, \]
where~$ \deg(\omega) $ and~$ \dom(\omega) $ denote the vertex degree and number of vertices dominating~$ \omega $, respectively.

We further show that~$ \tau $ is closed in~$ \vS $ if and only if~$ \omega $ is dominated by infinitely many vertices.

A question raised in~\cite{Profiles} asks whether for a~$ k $-tangle~$ \tau $ in a finite graph~$ G $ one can always find a set~$ X $ of vertices which decides~$ \tau $ by majority vote, in the sense that~$ (A,B)\in\tau $ if and only if~$ \abs{A\cap X}<\abs{B\cap X} $, for all~$ (A,B)\in\vS_k $. This problem is still open in general, although some process has been made recently~(see~\cite{ChristiansMasterarbeit,WeightedDeciders}). We establish an analogue in the infinite setting: we show that for an end~$ \omega $ of~$ G $ and its induced~$ k $-tangle~$ \tau\cap\vS_k $ in~$ G $, the existence of a finite set~$ X $ which decides~$ \tau\cap\vS_k $ in the above sense is equivalent to~$ \tau\cap\vS_k $ being a closed set in~$ \vS $. In this way the global property of being topologically closed can be linked to the local phenomenon of a finite set deciding the end tangle.

This paper is organized as follows:~\cref{sec:infinite_ends_notation} contains the basic definitions and some notation. Following that, in~\cref{sec:infinite_ends_S}, we recall the core concepts and results from~\cite{EndsAndTangles} that are relevant to our studies, including the topology defined on~$ \vS $. Finally, in~\cref{sec:infinite_ends_Sk}, we prove our main results~\cref{thm:categories} and~\cref{thm:vSk}. The first of these characterises the ends of a graph by the behaviour of their tangles, and the second shows that~$ \tau\cap\vS_k $ being a closed set in~$ \vS $ for a~$ k $-tangle~$ \tau $ induced by some end~$ \omega $ of~$ G $ is equivalent to both~$ \deg(\omega)+\dom(\omega)\ge k $ and to~$ \tau\cap\vS_k $ being decided by some finite set of vertices.

\section{Separations, tangles, and their topology}\label{sec:infinite_ends_notation}

Throughout this paper~$ G=(V,E) $ will be a fixed infinite graph. Let us recall the relevant definitions for tangles in graphs, and their extensions to infinite graphs. For any graph-theoretical notation not explained here we refer the reader to~\cite{DiestelBook}.

A {\em separation} of~$ G $ is a set~$ \{A,B\} $ with~$ A\cup B=V $ such that~$ G $ contains no edge between~$ A\sm B $ and~$ B\sm A $. We call such a set~$ \{A,B\} $ an {\em unoriented} separation with the two orientations~$ (A,B) $ and~$ (B,A) $. Informally we think of the oriented separation~$ (A,B) $ as {\em pointing towards}~$ B $ and {\em pointing away} from~$ A $. The {\em separator} of a separation~$ \{A,B\} $ is the set~$ A\cap B $.

The {\em order} of a separation~$ (A,B) $ or~$ \{A,B\} $ of~$ G $ is the cardinality~$ \abs{A\cap B} $ of its separator. For a cardinal~$ \kappa $ we write~$ S_\kappa=S_\kappa(G) $ for the set of all unoriented separations of~$ G $ of order~$ <\kappa $. If~$ S $ is a set of unoriented separations we write~$ \vS $ for the corresponding set of oriented separations, that is, the set of all separations~$ (A,B) $ with $ {\{A,B\}\in S} $. Consequently we write~$ \vS_\kappa $ for the set of all separations~$ (A,B) $ of~$ G $ with~$ \abs{A\cap B}<\kappa $.

If~$ S $ is a set of unoriented separations of~$ G $, an {\em orientation} of~$ S $ is a set~$ O\subseteq\vS $ such that~$ O $ contains precisely one of~$ (A,B) $ or~$ (B,A) $ for every~$ \{A,B\}\in S $. A {\em tangle of~$ S $ in~$ G $} is an orientation~$ \tau $ of~$ S $ such that there are no~$ (A_1,B_1)$,~$ (A_2,B_2)$, and~$ (A_3,B_3) $ in~$ \tau $ for which~$ G[A_1]\cup G[A_2]\cup G[A_3]=G $.

Properties of sets of separations of finite graphs, including their tangles, often generalize to sets of separations of infinite graphs but not always. Those sets of separations to which these properties tend to generalize can be identified, however: they are the sets of separations that are closed in a certain natural topology~\cite{ProfiniteASS}. Let us define this topology next. It is analogous to the topology of a profinite abstract separation system defined in~\cite{ProfiniteASS}.\footnote{Even though~$ \vS $ itself is not usually profinite in the sense of~\cite{ProfiniteASS}, the topology we define on~$ \vS $ is the subspace topology of~$ \vS $ as a subspace of the (profinite) system of all oriented separations of~$ G $, equipped with the inverse limit topology from~\cite{ProfiniteASS}.}

From here on we denote by $ {S=S_{\aleph_0}(G)} $ the set of all (unoriented) separations of~$ G $ of finite order. Thus~$ \vS=\vS_{\aleph_0}(G) $ is the set of all separations~$ (A,B) $ of~$ G $ with~$ A\cap B $ finite.

We define our topology on~$ \vS$ by giving it the following basic open sets. Pick a finite set~$ Z\sub V$ and an oriented separation~$ (A_Z,B_Z)$ of~$ G[Z]$. Then declare as open the set~$ O(A_Z,B_Z)$ of all~$ (A,B)\in\vS$ such that~$ A\cap Z = A_Z$ and $B\cap Z = B_Z$. We shall say that these~$ (A,B)$ {\em induce\/}~$ (A_Z,B_Z)$ on~$Z$, writing~$ (A_Z,B_Z)=: (A,B)\restricts Z$, and that~$ (A,B)$ and~$ (A',B')$ {\em agree on~$Z$} if~$ (A,B)\restricts Z = (A',B')\restricts Z$.

It is easy to see that the sets~$ O(A_Z,B_Z)$ do indeed form the basis of a topology on~$ \vS$. Indeed,~$ (A,B)\in\vS$ induces~$ (A_1,B_1)$ on~$Z_1$ and~$ (A_2,B_2)$ on~$Z_2$ if and only if it induces on~$ Z = Z_1\cup Z_2$ some separation~$ (A_Z,B_Z)$ which in turn induces~$ (A_i,B_i)$ on~$Z_i$ for both~$i$. Hence~$ O(A_1,B_1)\cap O(A_2,B_2)$ is the union of all these~$ O(A_Z,B_Z)$.

As we shall see, the intuitive property of tangles in finite graphs that they describe, if indirectly, some highly cohesive region of that graph -- however `fuzzy' this may be in terms of concrete vertices and edges -- will extend precisely to those tangles of~$ S $ that are closed in~$ \vS $.

\section{End tangles of~$ S $}\label{sec:infinite_ends_S}

We think of an oriented separation~$ (A,B)\in\vS $ as pointing towards~$ B $, or being oriented towards~$ B $. In the same spirit, given an end~$ \omega $ of~$ G $, we say that~$ (A,B) $ {\em points towards}~$ \omega $, and that~$ \omega $ {\em lives in}~$ B $, if some (equivalently: every) ray of~$ \omega $ has a tail in~$ B $. Furthermore, if~$ (A,B) $ points to an end~$ \omega $, then~$ (B,A) $ {\em points away from}~$ \omega $.

Clearly, for every end of~$ G $ and every~$ \{A,B\}\in S $, precisely one orientation of~$ \{A,B\} $ points towards that end. In this way, every end~$ \omega $ of~$ G $ defines an orientation of~$ S $ by orienting each separation in~$ S $ towards~$ \omega $:
\[ \tau=\tau_\omega\coloneqq \{(A,B)\in\vS\mid\tn{every ray of }\omega\tn{ has a tail in }B\} \]
It is easy to see~(\cite{EndsAndTangles}) that this is a tangle in~$ G $. We call it the {\em end tangle induced} on~$ S $ by the end~$ \omega $.

Note that every end tangle contains all separations of the form~$ (A,V) $ for finite~$ A\sub V $, and thus no separation of the form~$ (V,B) $. Furthermore, any two ends induce different end tangles. Our aim in this section is to recall from~\cite{EndsAndTangles} some properties of the end tangles of~$ S $ that we shall later extend to its subsets~$ S_k $. For the convenience of the reader, and also in order to correct an inessential but confusing error in~\cite{EndsAndTangles}, we repeat some of the material from~\cite{EndsAndTangles} here to make our presentation self-contained.

Let us first see an example of an end tangle that is not closed in~$ \vS $.

\begin{EX}\label{ex:Ray}
	{\em If~$ G$ is a single ray~$ v_0 v_1\dots$ with end~$ \omega$, say, then~$ \tau=\tau_\omega$~is not closed in~$ \vS$.}
	
	Indeed,~$ \tau$ contains~$ (\es,V)$, and hence does not contain~$(V,\es)$. But for every finite~$ Z\sub V$ the restriction~$ (Z,\es)$ of~$ (V,\es)$ to~$Z$ is also induced by the separation $ (\{v_0,\dots,v_n\}, \{v_n, v_{n+1},\dots\})\in\tau$ for every~$ n$ large enough that~$ Z\sub\{v_0,\dots,v_{n-1}\}$. So~$ (V,\es)\in\vS\sm\tau$ has no open neighbourhood in~$ \vS\sm\tau.\hfill \square$
\end{EX}

Here is an example of an end tangle that is closed in~$ \vS $. Unlike our previous example, it describes a highly cohesive part of~$ G $.

\begin{EX}\label{ex:Complete}
	{\em If~$ K\sub V$ spans an infinite complete graph in~$ G$, then
		\[ \tau = \{\,(A,B)\in\vS\mid K\sub B\,\}\tag{1}\label{eq:K} \]
		is a closed set in~$ \vS$.}
	
	We omit the easy proof. But note that~$ \tau $ is indeed an end tangle: it is induced by the unique end of~$ G $ which contains all the rays in~$ K $. %Jakob: This last sentence is new.
\end{EX}

Perhaps surprisingly, it is not hard to characterize the end tangles that are closed. They are all essentially like~\cref{ex:Complete}: we just have to generalize the infinite complete subgraph used appropriately. Of the two obvious generalizations, infinite complete minors~\cite{RSTminors} or subdivisions of infinite complete graphs~\cite{RST}, the latter turns out to be the right one.

Let~$ \kappa$ be any cardinal. A set of at least~$ \kappa $ vertices of~$ G $ is~$ (<\kappa) $-{\em inseparable} if no twoof them can be separated in~$ G $ by fewer than~$ \kappa $ vertices. A maximal~$ (<\kappa) $-inseparable set of vertices is a~$ \kappa $-{\em block}. For example, the branch vertices of a~$TK_\kappa$ are~$ (<\kappa) $-inseparable. Conversely:

\begin{LEM}\label{lem:TK}
	When~$ \kappa $ is infinite, every~$ (<\kappa) $-inseparable set of vertices in~$ G $ contains the branch vertices of some~$ TK_\kappa\sub G $.
\end{LEM}

\begin{proof}
	Let~$ K\sub V $ be~$ (<\kappa) $-inseparable. Viewing~$ \kappa $ as an ordinal we can find, inductively for all~$ \alpha<\kappa $, distinct vertices $ v_\alpha\in K $ and internally disjoint $ v_\alpha $-$ v_\beta $ paths in~$ G $ for all~$ \beta<\alpha $ that also have no inner vertices among those $ v_\beta $ or on any of the paths chosen earlier; this is because~$ \abs{K}\ge\kappa $, and no two vertices of~$ K $ can be separated in~$ G $ by the~$ <\kappa $ vertices used up to that time.
\end{proof}

The original statement of~\cref{lem:TK} in~\cite[Lemma~5.4]{EndsAndTangles} asserted that for infinite~$ \kappa $ a set~$ K\sub V $ is a~$ \kappa $-block in~$ G $ if and only if it is the set of branch vertices of some~$ TK_\kappa\sub G $. It turns out that both directions of that assertion were incorrect: the set of branch vertices of a~$ TK_\kappa\sub G $ is certainly~$ (<\kappa) $-inseparable, but might not be maximal with this property and hence not a~$ \kappa $-block. Conversely, if~$ K $ is a~$ \kappa $-block, there might not be a~$ TK_\kappa\sub G $ whose set of branch vertices is precisely~$ K $: if~$ \abs{K}>\kappa $ this is certainly not possible, but even if~$ \abs{K}=\kappa $ one might not be able to find a~$ TK_\kappa $ in~$ G $ whose branch vertices are all of~$ K $. If for instance the graph~$ G $ is a clique on~$ \kappa $ vertices that is missing exactly one edge, then~$ K=V(G) $ is a~$ \kappa $-block in~$ G $ but not the set of branch vertices of a~$ TK_\kappa\sub G $.

For the main theorem of this section we need one more observation:

\begin{LEM}\label{lem:profile}
	If~$ \tau $ is an end tangle of~$ G $ that contains~$ (A,B) $ and~$ (C,D) $ then~$ \tau $ also contains~$ {(A\cup C\,,\,B\cap D)} $. %Jakob: This lemma replaces Lemma xxxP, which said that abstract tangles are profiles.
\end{LEM}

\begin{proof}
	Observe first that~$ (A\cup C\,,\,B\cap D) $ is a separation of~$ G $ with finite order and thus lies in~$ \vS $. Moreover, if a ray of~$ G $ has a tail in~$ B $ and a tail in~$ D $, then that ray also has a tail in~$ B\cap D $. From this it follows that~$ (A\cup C\,,\,B\cap D)\in\tau $, as claimed. 
\end{proof}

We can now re-prove and slightly extend the characterization from~\cite{EndsAndTangles} of the tangles that are closed in~$ \vS $. Let us say that a set~$ {K\sub V}$ is an {\em absolute decider} for a tangle~$ \tau $ of~$ S $ if~$ \tau$ satisfies~\eqref{eq:K}.

\begin{THM}[\cite{EndsAndTangles}]\label{thm:blocks}
	For a tangle~$ \tau $ of~$ S $ the following are equivalent:
	\begin{itemize}
		\item[\tn{(i)}]~$ \tau $ is closed in~$ \vS $;
		\item[\tn{(ii)}]~$ \tau $ is absolutely decided by some set~$ K\sub V $;
		\item[\tn{(iii)}]~$ \tau $ is absolutely decided by an~$ {\aleph_0} $-block~$ K $.
	\end{itemize}
\end{THM}

\begin{proof}
	We will show~$ \tn{(iii)}\Rightarrow\tn{(ii)}\Rightarrow\tn{(i)}\Rightarrow\tn{(iii)} $. The first of these implications is clear.
	
	To see that~$ \tn{(ii)}\Rightarrow\tn{(i)} $ suppose that~$ \tau $ is a tangle of~$ S $ that is absolutely decided by some set~$ K\sub V$. Observe first that~$ K $ must be infinite. To show that~$ \tau$ is closed, we have to find for every~$ (A,B)\in\vS\sm\tau$ a finite set~$ Z\sub V$ such that no~$ (A',B')\in\vS$ that agrees with~$ (A,B)$ on~$Z$ lies in~$ \tau$. Since~$ (A,B)\notin\tau$ and hence~$ (B,A)\in\tau $ we have~$ K\sub A$; pick~$ z\in K\sm B$. Then every~$ (A',B')\in\vS$ that agrees with~$ (A,B)$ on~$ Z\coloneqq \{z\}$ also also lies in~$ \vS\sm\tau$, since~$ z\in A'\sm B'$ and this implies~$ K\not\sub B'$. %Jakob: The original~\cite{EndsAndTangles} has a typo at the very end of the paragraph:
	%Was:~$ K\!\notin B'$
	%Should be:~$ K\not\sub B'$
	
	To see that~$ \tn{(i)}\Rightarrow\tn{(iii)} $ let
	$$K\coloneqq  \bigcap\{\,B\mid (A,B)\in\tau\,\}.$$
	No two vertices in~$K$ can be separated by in~$ G$ by a finite-order separation: one orientation~$ (A,B)$ of this separation would be in~$ \tau$, which would contradict the definition of~$K$ since~$ A\sm B$ also meets~$K$. If~$ K$ is infinite, it will clearly be maximal with this property, and hence be an~$ {\aleph_0}$-block. This~$ {\aleph_0}$-block~$ K$ will be an absolute decider for~$ \tau$: by definition of~$ K$ we have~$ K\sub B$ for ever~$ (A,B)\in\tau$, while also every~$ (A,B)\in\vS$ with~$ K\sub B$ must be in~$ \tau$: otherwise~$ (B,A)\in\tau$ and hence~$ K\sub A$ by definition of~$K$, but~$ K\not\sub A\cap B$ because this is finite. Hence~$ \tau$ will be decided absolutely by an~$ {\aleph_0}$-block, as desired for this implication.%
	\footnote{Whether or not~$ \tau$ is closed in~$ \vS$ is immaterial; we just did not use this assumption.}
	
	It thus suffices to show that if~$ K$ is finite then~$ \tau$ is not closed in~$ \vS$, which we shall do next.
	
	Assume that~$ K$ is finite. We have to find some~$ (A,B)\in\vS\sm\tau$ that is a limit point of~$ \tau$, i.e., which agrees on every finite~$ Z\sub V$ with some~$ (A',B')\in\tau$. We choose~$ (A,B)\coloneqq  (V,K)$, which lies in~$ \vS\sm\tau $ since~$ (K,V)\in\tau $.
	
	To complete our proof as outlined, let any finite set~$ Z\sub V$ be given. For every~$ z\in Z\sm K$ choose~$ (A_z,B_z)\in\tau$ with~$ z\in A_z\sm B_z$: this exists, because~$ {z\notin K}$. Since~$ (K,V)\in\tau $, by~\cref{lem:profile} we have~$ (A',B')\in\tau$ for
	\[ A'\coloneqq  K\cup\bigcup_{z\in Z\sm K} A_z\quad\tn{and}\quad B'\coloneqq  V\cap\bigcap_{z\in Z\sm K} B_z\,. \]
	As desired,~$ (A',B')\restricts Z = (A,B)\restricts Z$ (which is~$ (Z,Z\cap K)$, since~$ (A,B) = (V,K)$): every~$ z\in Z\sm K$ lies in some~$ A_z$ and outside that~$ B_z$, so~$ z\in A'\sm B'$, while every~$ z\in Z\cap K$ lies in~$ K\sub A'$ and also, by definition of~$K$, in every $B_z$ (and hence in~$ B'$), since~$ (A_z,B_z)\in\tau$.
\end{proof}

This proof of~\cref{thm:blocks} concludes our exposition of material from~\cite{EndsAndTangles}.

\section{End tangles of~$ S_k $}\label{sec:infinite_ends_Sk}

In~\cref{thm:blocks} we characterised those tangles of~$ S $ that are closed in~$ \vS=\vS_{\aleph_0} $. However even if an end tangle of~$ S $ is not closed in~$ \vS $ its restrictions to~$ \vS_k\sub\vS $ may still be closed in~$ \vS $ for some or even all~$ k\in\N $. The set of values of~$ k $ for which this is the case may hold some combinatorial information about that tangle and the end that is inducing it. In the remainder of this paper we will study the connection between the combinatorial properties of an end and the set of values of~$ k $ for which its end tangle's restriction to~$ \vS_k $ is closed. In particular we will show how these values of~$ k $ relate to the vertex degree and number of domination vertices of an end. Let us set up some terminology for this.

Observe first that the set~$ \vS_k $ is a closed subset of~$ \vS $ for each~$ k\in\N $. Given an end~$ \omega $ of~$ G $ with end tangle~$ \tau=\tau_\omega $, clearly~$ \tau\cap\vS_k $ is a tangle of~$ S_k $; we call it the {\em tangle of~$ S_k $ induced by~$ \omega $}. Where we say that this tangle of~$ S_k $ is~\emph{closed}, we mean that it is closed in~$ \vS_k $ or, equivalently, in~$ \vS $.

%Let us classify and the ends of~$ G $ by the values of~$ k $ for which the tangles of~$ S_k $ they induce are closed or not.
For~$ \ell<k $ the set~$ \vS_\ell $ is closed in~$ \vS_k $, and hence any end inducing a closed tangle of~$ S_k $ also induces a closed tangle of~$ S_\ell $. This motivates the following definition: for an end~$ \omega $ of~$ G $ let the~\emph{cohesion} of~$ \omega $ be
\[ \coh(\omega)\coloneqq \sup\menge{\kappa\le\aleph_1\mid\tau_\omega\cap\vS_k\tn{ is closed in }\vS\tn{ for all }k<\kappa}\,. \]
The ends of~$ G $ then fall into three distinct categories: we say that an end~$ \omega $ has
\begin{itemize}
	\item[(i)] \emph{infinite cohesion} if~$ \coh(\omega)=\aleph_1 $, i.e. if~$ \tau_\omega $ is closed;
	\item[(ii)] \emph{unbounded cohesion} if~$ \coh(\omega)=\aleph_0 $, i.e. if~$ \tau_\omega $ is not closed but~$ \tau_\omega\cap\vS_k $ is for all~$ k\in\N $;
	\item[(iii)] \emph{bounded cohesion}~$ \coh(\omega)\in\N $ otherwise.
\end{itemize}

If an end does not have infinite cohesion we say that it has~\emph{finite cohesion}.

Phrased in this language~\cref{thm:blocks} characterised the ends of infinite cohesion: they are those whose end tangle is decided absolutely by an~$ \aleph_0 $-block. In fact, as we shall show later, the latter is equivalent to that end being dominated by infinitely many vertices. Our goal in the remainder of this section is to obtain similar characterisations for ends of finite cohesion, both in terms of vertex sets deciding the respective tangle and in terms of combinatorial parameters of that end.

Let us first see examples of ends belonging to the third and second category, respectively:

\begin{EX}\label{ex:Ray2}
	Let~$ G $ be as in~\cref{ex:Ray}, that is, a single ray~$ v_0 v_1\dots$ with end~$ \omega$. The same argument as in~\cref{ex:Ray} shows that~$ \tau\cap\vS_k $ is not closed in~$ \vS_k $ for~$ k\ge 2 $. However,~$ \omega $ does induce a closed tangle of~$ S_1 $: the set~$ \tau\cap\vS_1=\{(\emptyset,V)\} $ is closed in~$ \vS_1 $.
\end{EX}

\begin{EX}\label{ex:grid}
	Let~$ G $ be the infinite grid,~$ \omega $ the unique end of~$ G $ and~$ \tau $ the tangle induced by~$ \omega $ in~$ S=S(G) $. Since~$ G $ is locally finite it does not contain an $ {\aleph_0} $-block. Therefore~$ \tau $ cannot be decided absolutely by an $ {\aleph_0} $-block and is thus not closed in~$ \vS $ by~\cref{thm:blocks}. However for every~$ k\in\N $ it is easy to see that~$ \tau\cap\vS_k $ is closed: indeed for fixed~$ k\in\N $ the size of~$ A $ is bounded in terms of~$ k $ for every~$ (A,B)\in\tau\cap\vS_k $. Thus for every~$ (A',B')\in\vS_k\sm\tau $ any~$ Z\sub V(G) $ with~$ \abs{Z\cap A'} $ sufficiently large witnesses that~$ (A',B') $ does not lie in the closure of~$ \tau $: if~$ Z $ is large enough that~$ \abs{Z\cap A'}>\abs{A} $ for every~$ (A,B)\in\tau\cap\vS_k $ then no~$ (A,B)\in\tau\cap\vS_k $ will agree with~$ (A',B') $ on~$ Z $.
\end{EX}

\cref{ex:grid} shows that the end tangle~$ \tau $ of the infinite grid is not decided absolutely by a set~$ X\sub V $ in the sense of~\eqref{eq:K}. In fact this is true even for~$ \tau\cap\vS_k $ for~$ k\ge 5 $: there can be no set~$ X\sub V $ with~$ X\sub B $ for all~$ (A,B)\in\tau\cap\vS_k $ since for each~$ x\in X $ the separation~$ (\{x\}\cup N(x)\,,\,V\sm\{x\}) $ lies in~$ \tau\cap\vS_k $. Thus even the tangles of~$ S_k $ that~$ \tau $ induces do not have absolute deciders. However they come reasonably close to it: for~$ X=V $ and any~$ (A,B)\in\tau $ the \emph{relative} majority of~$ X $ lies in~$ B $, that is, we have~$ \abs{A\cap X}<\abs{B\cap X} $. In fact for any fixed~$ k\in\N $ every finite set~$ X\sub V $ that is at least twice as large as~$ \max\{\abs{A}\mid(A,B)\in\tau\cap\vS_k\} $ has the property that~$ \abs{A\cap X}<\abs{B\cap X} $ for each~$ (A,B)\in\tau\cap\vS_k $. Therefore even though no~$ \tau\cap\vS_k $ has an absolute decider we can for each~$ k\in\N $ find a (finite) `relative decider' of~$ \tau\cap\vS_k $.

Let us make the above observation formal. For an end tangle~$ \tau $ of~$ G $ let us call a set~$ X\sub V $ a {\em relative decider} for~$ \tau~$ (resp., for~$ \tau\cap\vS_k $) if we have~$ \abs{A\cap X}<\abs{B\cap X} $ for every~$ (A,B)\in\tau~$ (resp.,~$ (A,B)\in\tau\cap\vS_k $). Thus if we have a relative decider for an end tangle~$ \tau $ then given a separation~$ (A,B)\in\vS $ the decider set tells us which of~$ (A,B) $ and~$ (B,A) $ lies in~$ \tau $ by a simple majority vote. Clearly each absolute decider of a tangle is also a relative decider.

If~$ \omega $ is an end of infinite cohesion then by~\cref{thm:blocks} its end tangle~$ \tau=\tau_\omega $ has an (infinite) absolute decider. In analogy with this we shall show that if an end tangle~$ \tau $ is closed in~$ \vS_k $ then~$ \tau $ has a finite relative decider. Such a finite relative decider can be thought of as a local encoding of the tangle or a local witness to the tangle being closed in~$ \vS_k $.

In contrast it is easy to see that every end tangle has an infinite relative decider:

\begin{PROP}\label{prop:infinitedeciderset}
	For any end~$ \omega $ of~$ G $ the end tangle~$ \tau $ induced by~$ \omega $ has an infinite relative decider.
\end{PROP}

\begin{proof}
	For any ray~$ R\in\omega $ its vertex set~$ V(R) $ is a relative decider for~$ \tau $.
\end{proof}

Since no end tangle of~$ S $ can have a finite decider (relative or absolute) studying the existence of finite deciders for end tangles is thus only interesting for the tangles' restrictions to some~$ \vS_k $.

We shall complement this local witness of a given end tangle being closed with a more global type of witness: the vertex degree and number of vertices dominating that end. The latter two are well-studied parameters of ends; let us recall their definitions.

The {\em (vertex) degree}~$ \deg(\omega) $ of an end~$ \omega $ of~$ G $ is the largest size of a family of pairwise disjoint~$ \omega $-rays\footnote{Here our notation deviates from that in~\cite{DiestelBook}, where $ d(\omega) $ is used for the degree of~$ \omega $.}. A vertex $ v\in V $ {\em dominates} an end~$ \omega $ if it sends infinitely many disjoint paths to some (equivalently: to each) ray in~$ \omega $. We write~$ \dom(\omega) $ for the number of vertices of~$ G $ which dominate~$ \omega $. An end~$ \omega $ is {\em undominated} if~$ \dom(\omega)=0 $; it is {\em finitely dominated} if finitely many (including zero) vertices of~$ G $ dominate~$ \omega $; and, finally,~$ \omega $ is {\em infinitely dominated} if~$ \dom(\omega)=\infty $.

We are now ready to formally state the connection between the cohesion~$ \coh(\omega) $ of an end~$ \omega $ of~$ G $ and these parameters~$ \deg(\omega) $ and~$ \dom(\omega) $:

\begin{restatable}{THM}{ThmCategories}\label{thm:categories}
	Let~$ \omega $ be an end of~$ G $. Then the following statements hold:
	\begin{itemize}
		\item[\tn{(i)}]~$ \omega $ has infinite cohesion if and only if~$ \dom(\omega)=\infty $.
		\item[\tn{(ii)}]~$ \omega $ has unbounded cohesion if and only if~$ \deg(\omega)=\infty $ and~$ \dom(\omega)<\infty $.
		\item[\tn{(iii)}]~$ \omega $ has bounded cohesion~$ \coh(\omega)=k $ if and only if $ \deg(\omega)+\dom(\omega)=k-1 $.
	\end{itemize}
\end{restatable}

\cref{thm:categories} will be a consequence of~\cref{thm:blocks} and the following theorem, which characterises for which~$ k\in\N $ a given end tangle's restrictions are closed and makes the connection to relative decider sets:

\begin{restatable}{THM}{EndsChar}\label{thm:vSk}
	Let~$ \tau $ be the end tangle induced by an end~$ \omega $ of~$ G $ and let~$ k\in\N $. Then the following are equivalent:
	\begin{itemize}
		\item[\tn{(i)}]~$ \tau\cap\vS_k $ is closed;
		\item[\tn{(ii)}]~$ \deg(\omega)+\dom(\omega)\ge k $;
		\item[\tn{(iii)}]~$ \tau\cap\vS_k $ has a finite relative decider;
		\item[\tn{(iv)}]~$ \tau\cap\vS_k $ has a relative decider of size exactly~$ k $.
	\end{itemize}
\end{restatable}

Let us first derive~\cref{thm:categories} from~\cref{thm:blocks} and~\cref{thm:vSk}:

\begin{proof}[Proof of~\cref{thm:categories}]
	Let~$ \omega $ be an end of~$ G $ and~$ \tau=\tau_\omega $ its end tangle. We will show (i) using~\cref{thm:blocks} and derive (iii) from~\cref{thm:vSk}. Then (ii) is an immediate consequence of the other two.
	
	To see that (i) holds let us first suppose that~$ \omega $ has infinite cohesion, i.e. that~$ \tau $ is closed in~$ \vS $. Then by~\cref{thm:blocks}~$ \tau $ is decided absolutely by an $ {\aleph_0} $-block~$ K\sub V $. It is easy to see that each vertex of~$ K $ dominates~$ \omega $. Since~$ K $ is infinite we thus indeed have~$ \dom(\omega)=\infty $.
	
	For the converse suppose that~$ \omega $ is infinitely dominated and let us show that it has infinite cohesion. For each separation~$ (A,B)\in\tau $ each vertex dominating~$ \omega $ must be contained in~$ B $. Therefore~$ K\coloneqq \bigcap\{B\mid(A,B)\in\tau\} $ is infinite and thus, as seen in the proof of~\cref{thm:blocks}, an absolute decider for~$ \tau $, which is hence closed in~$ \vS $. %Jakob: One could easily prove closed from~$ \dom=\infty $ directly, but that would be slightly longer than appealing to~\cref{thm:blocks}.
	
	Claim (iii) is a direct consequence of the definition of~$ \coh(\omega) $ and the equivalence of the first two statements of~\cref{thm:vSk}.
\end{proof}

We will conclude this section by proving~\cref{thm:vSk}. For this we shall need two preparatory lemmas. The first lemma can be seen as an analogue of Menger's Theorem between a vertex set and an end. Given a set~$ X\sub V $ and an end~$ \omega $ we say that~$ F\sub V $ {\em separates~$ X $ from~$ \omega $} if every~$ \omega $-ray which meets~$ X $ also meets~$ F $. %Jakob: This is a slightly cheeky way of defining separating vertices from an end.	Also, I'd really like to use~$ S $ for the separating set here, but that's already used for~$ S=S(G) $ and~$ \vS $, and we shouldn't overload this variable.

\begin{LEM}\label{lem:sequence}
	Let~$ \omega $ be an undominated end of~$ G $ and~$ X\sub V $ a finite set. The largest size of a family of disjoint~$ \omega $-rays which start in~$ X $ is equal to the smallest size of a set~$ T\sub V $ separating~$ X $ from~$ \omega $.
\end{LEM}

\begin{proof}
	Let~$ T $ be a set separating~$ X $ from~$ \omega $ of minimal size. Clearly a family of disjoint~$ \omega $-rays which all start in~$ X $ cannot be larger than~$ T $ since each ray in that family must meet~$ T $. So let us show that we can find a family of~$ \abs{T} $ disjoint~$ \omega $-rays starting in~$ X $.
	
	Observe that since~$ \omega $ is undominated we can find for each vertex~$ v\in V $ a finite set~$ T_v\sub V\sm\{v\} $ which separates $ v $  from~$ \omega $. Thus, for every finite set~$ Y\sub V $ we can find a finite set in~$ V\sm Y $ separating~$ Y $ from~$ \omega $: for instance, the set~$ \bigcup\{T_v\sm Y\mid v\in Y\} $.
	
	Pick a sequence of finite sets~$ T_n\sub V $ inductively by setting~$ T_0\coloneqq T $ and picking as~$ T_n $ a set of minimal size with the property that~$ T_{n-1}\cap T_n=\emptyset $ and that~$ T_n $ separates~$ T_{n-1} $ from~$ \omega $; these sets exist by the above observation. Let~$ C_n $ be the component of~$ G-T_n $ that contains~$ \omega $. Clearly~$ T_n\sub C_{n-1} $.
	
	We claim that~$ C\coloneqq \bigcap_{n\in\N}C_n=\emptyset $. To see this, consider any $ v\in C $ and a shortest $ v $--$ X $ path in~$ G $. This path must pass through~$ T_n $ for every $ n\in\N $, which is impossible since the separators~$ T_n $ are pairwise disjoint. Therefore~$ C $ must be empty.
	
	By the minimality of each~$ T_n $, the sets~$ T_n $ are of non-decreasing size, and furthermore Menger's Theorem yields a family of~$ \abs{T_n} $ many disjoint paths between~$ T_n $ and~$ T_{n+1} $ for each $ n\in\N $, as well as~$ \abs{T_0} $ many disjoint paths between~$ X $ and~$ T_0 $. By concatenating these paths we obtain a family of~$ \abs{T_0}=\abs{T} $ many rays starting in~$ X $. To finish the proof we just need to show that these rays belong to~$ \omega $. To see this, let~$ \omega' $ be another end of~$ G $, and~$ T' $ a finite set separating~$ \omega $ and~$ \omega' $. Since~$ C=\emptyset $ we have~$ C_n\cap T'=\emptyset $ for sufficiently large $ n $, which shows that the rays constructed do not belong to~$ \omega' $ and hence concludes the proof.
\end{proof}

An immediate consequence of~\cref{lem:sequence} is that for an undominated end~$ \omega $ every finite set~$ X\sub V $ can be separated from~$ \omega $ by at most~$ \deg(\omega) $ many vertices. In fact we can state a slightly more general corollary:

\begin{COR}\label{cor:deg+dom}
	Let~$ \omega $ be a finitely dominated end of~$ G $ and~$ X\sub V $ a finite set. Then~$ X $ can be separated from~$ \omega $ by some~$ T\sub V $ with~$ \abs{T}\le\deg(\omega)+\dom(\omega) $.
\end{COR}

\begin{proof}
	Let~$ D $ be the set of vertices dominating~$ \omega $ and consider the graph~$ G'\coloneqq G-D $ and the set~$ X'\coloneqq X\sm D $. By~\cref{lem:sequence} there is a set~$ T'\sub V(G') $ of size at most~$ \deg(\omega) $ separating~$ X' $ from~$ \omega $ in~$ G' $. Set~$ T\coloneqq T'\cup D $. Then~$ T $ separates~$ X $ from~$ \omega $ in~$ G $ and has size~$ \abs{T}=\abs{T'}+\abs{D}\le\deg(\omega)+\dom(\omega) $. %Jakob: This corollary opens the can of (small) worms that are ends of~$ G-D $, but this shouldn't be a problem since~$ D $ is finite.
\end{proof}

The second lemma we shall need for our proof of~\cref{thm:vSk} roughly states that for an end of high degree we can find a large family of disjoint rays of that end whose set of starting vertices is highly connected in~$ G $, even after removing the tails of these rays:

\begin{LEM}\label{lem:paths}
	Let~$ \omega $ be an end of~$ G $ and $ {k\le\deg(\omega)+\dom(\omega)} $. Then there are a set~$ X\sub V $ of~$ k $ vertices and a set~$ \cR $ of disjoint~$ \omega $-rays with the following properties: every vertex in~$ X $ is either the start-vertex of a ray in~$ \cR $, or dominates~$ \omega $ and does not lie on any~$ R\in\cR $, and furthermore for any two sets~$ A,B\sub X $ there are~$ \min(\abs{A},\abs{B}) $ many disjoint~$ A $--$ B $-paths in~$ G $ whose internal vertices meet no ray in~$ \cR $ and no vertex of~$ X $.
\end{LEM}

\begin{proof}
	Pick a set~$ D $ of vertices dominating~$ \omega $ and a set~$ \cR $ of disjoint~$ \omega $-rays not meeting~$ D $ such that~$ \abs{D}+\abs{\cR}=k $; we shall find suitable tails of the rays in~$ \cR $ such that their starting vertices together with~$ D $ are the desired set~$ X $.
	
	Using the fact that the vertices in~$ D $ dominate~$ \omega $ and that the rays in~$ \cR $ belong to~$ \omega $, we can pick for each pair~$ x_1,x_2 $ of elements of~$ D\cup\cR $ an~$ x_1 $--$ x_2 $-path in~$ G $ in such a way that these paths are pairwise disjoint with the exception of possibly having a common end-vertex in~$ D $. Let~$ \cP $ be the set of these paths. Now for each ray in~$ \cR $ pick a tail of that ray which avoids all the paths in~$ \cP $. Let~$ \cR' $ be the set of these tails and~$ X $ the union of their starting vertices and~$ D $. We claim that~$ X $ and~$ \cR' $ are as desired.
	
	To see this, let us show that for any sets~$ A,B\sub X $ we can find~$ \min(\abs{A},\abs{B}) $ many disjoint~$ A $--$ B $-paths in~$ G $ whose internal vertices avoid~$ D $ as well as~$ V(R') $ for every~$ R'\in\cR' $. Clearly it suffices to show this for disjoint sets~$ A,B $ of equal size. So let~$ A,B\sub X $ be two disjoint sets with $ n\coloneqq \abs{A}=\abs{B} $ and let~$ \cR_{A,B} $ be the set of all rays in~$ \cR $ that contain a vertex from~$ A $ or~$ B $. For each pair~$ (a,b)\in A\times B $ there is a unique path~$ P\in\cP $ such that each of its end-vertices either is $ a $ or $ b~$ (if $ a\in D $ or $ b\in D $) or lies on a ray in~$ \cR_{A,B} $ which contains $ a $ or~$ b $; let~$ P_{a,b} $ be the $ a $--$ b $-path obtained from~$ P $ by extending it, for each of its end-vertices that is not either $ a $ or $ b $, along the corresponding ray in~$ \cR $ up to $ a $ or $ b $. Let~$ \cP_{A,B} $ be the set of all these paths~$ P_{a,b} $. Note that the internal vertices of each path~$ P_{a,b}\in\cP_{A,B} $ meet none of the rays in~$ \cR' $ or vertices in~$ X $.
	
	We claim that~$ A $ and~$ B $ cannot be separated by fewer than $ n=\abs{A} $ vertices in
	\[G'\coloneqq \bigcup_{P_{a,b}\in\cP_{A,B}}P_{a,b}\,;\]
	the claim will then follow from Menger's Theorem. So suppose that some set~$ T\sub V(G') $ of size less than $ n $ is given. Let~$ x $ and $ y $ be the number of vertices in~$ A $ and~$ B $, respectively, whose ray in~$ \cR_{A,B} $ does not meet~$ T $. There are~$ xy $ many paths in~$ \cP_{A,B} $ between these vertices in~$ A $ and~$ B $. Since these paths are disjoint outside their corresponding ray segments, each vertex of~$ T $ can lie on at most one of them. Thus if~$ xy\ge n $ there must be a~$ T $-avoiding path in~$ \cP_{A,B} $ whose end-vertices' rays in~$ \cR_{A,B} $ also do not meet~$ T $.
	
	Since~$ x+y\ge n+1 $ we have~$ xy\ge x(n+1-x) $. The right-hand side of this inequality, as a function of~$ x $ with domain $ [n-1] $, is minimized by taking~$ x=1 $, wherefore it evaluates to~$ n $. Thus~$ xy\ge n $, which shows that~$ T $ does not separate~$ A $ and~$ B $ in~$ G' $.
	
	We can thus apply Menger's Theorem to obtain $ n $ disjoint~$ A $--$ B $-paths in~$ G' $, which are the desired disjoint paths in~$ G $ whose internal vertices avoid the rays in~$ \cR' $ and vertices in~$ X $: the only vertices that are contained both in~$ V(G') $ as well as in either~$ X $ or a ray from~$ \cR' $ are vertices from~$ A $ or~$ B $, which cannot be internal vertices of the $ n=\abs{A}=\abs{B} $ disjoint~$ A $--$ B $-paths.
\end{proof}

We are now ready to prove~\cref{thm:vSk}:

\EndsChar*

\begin{proof}
	We will show~$ \tn{(i)}\Rightarrow\tn{(ii)}\Rightarrow\tn{(iv)}\Rightarrow\tn{(iii)}\Rightarrow\tn{(i)} $.
	
	To see that~$ \tn{(i)}\Rightarrow\tn{(ii)} $, let us suppose that $ {\deg(\omega)+\dom(\omega)<k} $ and show that~$ \tau\cap\vS_k $ is not closed. Let~$ D $ be the set of dominating vertices of~$ \omega $. Since~$ \abs{D}=\dom(\omega)<k $ the separation~$ (V,D) $ lies in~$ \vS_k $. By definition of~$ \tau $ we have~$ (D,V)\in\tau $ and~$ (V,D)\notin\tau $. Thus it suffices to show that~$ (V,D) $ lies in the closure of~$ \tau\cap\vS_k $ in~$ \vS_k $. This will be the case if for every finite set~$ X\sub V $ there is a separation~$ (A,B)\in\tau\cap\vS_k $ with~$ (A,B)\proj X=(V,D)\proj X $.
	
	So let~$ X $ be a finite subset of~$ V $. By~\cref{cor:deg+dom} some set~$ T $ of at most~$ {\deg(\omega)+\dom(\omega)} $ vertices separates~$ X $ from~$ \omega $. Let~$ C $ be the component of~$ G-T $ containing~$ \omega $. We define the separation~$ (A,B) $ by setting~$ A\coloneqq V\sm C $ and~$ B\coloneqq T\cup C $. Then~$ (A,B) $ is a separation of~$ G $ with
	\[ \abs{A\cap B}=\abs{T}\le\deg(\omega)+\dom(\omega)<k\,, \]
	giving~$ (A,B)\in\vS_k $. In fact~$ (A,B) $ lies in~$ \tau\cap\vS_k $ since~$ \omega $ lives in~$ B $. Furthermore we have~$ X\sub A $ and~$ D\sub B $ since no vertex dominating~$ \omega $ can be separated from~$ \omega $ by~$ T $. Therefore~$ (A,B)\proj X=(X,X\cap D)=(V,D)\proj X $, showing that~$ (V,D) $ lies in the closure of~$ \tau\cap\vS_k $ in~$ \vS_k $.
	
	Let us now show that~$ \tn{(ii)}\Rightarrow\tn{(iv)} $. So let us assume that $ {\deg(\omega)+\dom(\omega)\ge k} $. Then by~\cref{lem:paths} we find a set~$ X\sub V $ of size~$ k $ and a family~$ \cR $ of~$ \omega $-rays such that every vertex of~$ X $ either dominates~$ \omega $ or is the start-vertex of a ray in~$ \cR $, and such that for any~$ A,B\sub X $ we can find~$ \min(\abs{A},\abs{B}) $ many disjoint~$ A $--$ B $-paths in~$ G $ whose internal vertices meet neither~$ X $ nor any ray in~$ \cR $.
	
	We claim that~$ X $ is the desired relative decider for~$ \tau\cap\vS_k $. To see this let~$ (A,B) $ be any separation in~$ \tau\cap\vS_k $; we need to show that~$ \abs{A\cap X}<\abs{B\cap X} $. Let us write~$ X_{A\sm B}\coloneqq (A\sm B)\cap X $ and~$ X_{B\sm A}\coloneqq (B\sm A)\cap X $ as well as~$ X_{A\cap B}\coloneqq (A\cap B)\cap X $. It then suffices to prove~$ \abs{X_{A\sm B}}<\abs{X_{B\sm A}} $.
	
	So suppose to the contrary that~$ \abs{X_{A\sm B}}\ge\abs{X_{B\sm A}} $. Note first that no vertex in~$ X_{A\sm B} $ dominates~$ \omega $ as witnessed by the finite-order separation $ {(A,B)\in\tau\sub S_k} $. Therefore, for every vertex in~$ X_{A\sm B} $, we have a ray in~$ \cR $ starting at that vertex. Each of those disjoint rays must pass through the separator~$ A\cap B $, and none of them hits~$ X_{A\cap B} $. Furthermore by~\cref{lem:paths} there are~$ \abs{X_{B\sm A}} $ many disjoint~$ X_{A\sm B} $--$ X_{B\sm A} $-paths whose internal vertices avoid~$ \cR $ and~$ X $. These paths, too, must pass the separator~$ A\cap B $ without meeting~$ X_{A\cap B} $ or any of the rays above. Thus we have
	\[\abs{A\cap B}\ge\abs{X_{A\cap B}}+\abs{X_{A\sm B}}+\abs{X_{B\sm A}}=\abs{X}=k,\]%		Jakob: We have~$ |X\cap A\cap B|$ many vertices from~$ X$ in~$ A\cap B$ as well as~$ |X_{A\sm B}|$ many rays and~$ |X_{B\sm A}|$ many disjoint paths passing through~$ A\cap B$, all of them using different vertices of~$ A\cap B$.
	a contradiction since~$ (A,B)\in S_k $ and hence~$ \abs{A\cap B}<k $. Therefore we must have~$ \abs{X_{A\sm B}}<\abs{X_{B\sm A}} $, which immediately implies~$ \abs{A\cap X}<\abs{B\cap X} $.
	
	Finally, let us show that~$ \tn{(iii)}\Rightarrow\tn{(i)} $. So let~$ X\sub V $ be a finite relative decider for~$ \tau\cap\vS_k $. We need to show that no~$ (A,B)\in S_k\sm\tau $ lies in the closure of~$ \tau\cap\vS_k $. For this let~$ (A,B)\in S_k\sm\tau $ be given; then~$ X $ witnesses that~$ (A,B) $ does not lie in the closure of~$ \tau $. To see this let any~$ (C,D)\in\tau $ be given. Since~$ X $ is a relative decider for~$ \tau $ we have~$ \abs{C\cap X}<\abs{D\cap X} $, and since~$ (A,B)\notin\tau $ we have~$ \abs{A\cap X}\ge\abs{B\cap X} $. Therefore~$ (A,B) $ and~$ (C,D) $ do not agree on the finite set~$ X $, which thus witnesses that~$ (A,B) $ does not lie in the closure of~$ \tau\cap\vS_k $ in~$ S_k $.
\end{proof}

Note that in our proof above that~$ \tn{(iii)} $ implies~$ \tn{(i)} $ we did not make use of the assumption that the tangle~$ \tau $ is an end tangle: indeed every orientation of~$ S_k $ that has a finite relative decider is closed in~$ \vS_k $.

For an end tangle~$ \tau $ that is closed in~$ \vS $ we can say slightly more about its restrictions' relative deciders: for every~$ k\in\N $ the restriction~$ \tau\cap\vS_k $ has a relative decider of size exactly~$ k $ which is a~$ (<k) $-inseparable set. Finding these~$ (<k) $-inseparable decider sets is straightforward: such an end tangle~$ \tau $ is decided absolutely by an $ {\aleph_0} $-block~$ K $ by~\cref{thm:blocks}, and every subset~$ X\sub K $ of size~$ k $ is a~$ (<k) $-inseparable relative decider for~$ \tau\cap\vS_k $. (Every such set is an absolute decider, in fact.) However, having a~$ (<k) $-inseparable decider for~$ \tau\cap\vS_k $ for all~$ k\in\N $ is not a characterizing property for the closed end tangles of~$ G $:

\begin{EX}\label{ex:k-inseparable}
	For $ n\in\N $ let~$ K^n $ be the complete graph on $ n $ vertices. Let~$ G $ be the graph obtained from a ray~$ R=v_1v_2\dots $ by replacing each vertex $ v_n $ with the complete graph~$ K^n $, making each vertex from the~$ K^n $ replacing~$ v_n $ adjacent to all vertices from the~$ K^{n+1} $ replacing~$ v_{n+1} $. Then~$ G $ has a unique end~$ \omega $; let~$ \tau $ be the end tangle induced by~$ \omega $. Since~$ \omega $ is undominated~$ \tau $ is not closed in~$ \vS=\vS(G) $ by~\cref{thm:categories}. However, for every~$ k\in\N $, the tangle~$ \tau\cap\vS_k $ has a~$ (<k) $-inseparable absolute decider of size~$ k $: the clique~$ K^k $ which replaced the vertex $ v_k $ of~$ R $ is such a~$ (<k) $-inseparable decider.
\end{EX}

\vspace{1cm}
\noindent
\begin{minipage}{\linewidth}
	\raggedright\small
	\textbf{Jakob Kneip},
	Universität Hamburg,
	Hamburg, Germany \\
	\texttt{jakob.kneip@studium.uni-hamburg.de}
\end{minipage}

\end{document}